\documentclass[a4paper,10pt]{article}
\usepackage[all]{xy}

\usepackage{minitoc}
\usepackage{graphicx}
\usepackage[latin1]{inputenc}
\usepackage{amsmath,amsthm}
\usepackage{amsfonts,amssymb}
\usepackage{bigdelim}
\usepackage{multirow}
\usepackage{enumerate}
\usepackage{stmaryrd}
\usepackage{authblk}
\usepackage{float}
\usepackage{pstricks}
\usepackage[ae]{}
\newcommand{\E}{\mathbb{E}}

\newcommand{\R}{\mathbb{R}}

\def\={{\;\mathop{=}\limits^{\text{(law)}}\;}}

\newtheorem{theorem}{Theorem}[section]
\newtheorem{prop}[theorem]{Proposition}

\newtheorem{definition}[theorem]{Definition}

\theoremstyle{definition}
\newtheorem{remark}[theorem]{Remark}

\newtheorem{example}[theorem]{Example}

\numberwithin{equation}{section}

\usepackage[english]{babel}

\title{Weak decreasing stochastic order}
\author[]{Antoine-Marie Bogso\thanks{ambogso@uy1.uninet.cm, ambogso@gmail.com; Phone: (+237)652620452}}
\author[]{Patrice Takam Soh\thanks{ptakam@yahoo.fr; Phone: (+237)699299181}}
\affil[]{University of Yaoundé I, Department of Mathematics, P.O. Box 812 Yaoundé, Cameroon}
\date{}

\begin{document}
 
\maketitle
 
\begin{abstract}\noindent
We introduce the notion of weak decreasing stochastic (WDS) ordering for real-valued processes with negative means, which, to our knowledge, has not been studied before. Thanks to Madan-Yor's argument, it follows that the WDS ordering is a necessary and sufficient condition for a process with negative mean to be embeddable in a standard Brownian motion by the Cox and Hobson extension of the Azéma-Yor algorithm. Since the decreasing stochastic order is stronger than the WDS order, then, for every stochastically non-decreasing family of probability measures with densities, the Cox-Hobson   stopping times provide an associated Markov process. The quantile process associated to a stochastically non-decreasing process is not necessarily Markovian.\\ 
\text{}
\text{}\\
\textbf{subclass MSC:} 60E15, 60G44, 60J25.
\end{abstract}

\section{Introduction}
\label{intro}
 
We consider a new stochastic ordering for probability measures with negative means, namely the weak decreasing stochastic (WDS) ordering which is related to the usual stochastic and the increasing convex orders. We recall that a family of probability measures $\mu=(\mu_t,t\in\R_+)$ is said to be non-decreasing in {\it the usual stochastic order} if, for every $0\leq s\leq t$ and every non-decreasing function $\phi$ such that $\int_{\R}\phi(y)\mu_s(dy)$ and $\int_{\R}\phi(y)\mu_t(dy)$ exist,
\begin{equation}\label{eq:DcrStochOrder}
\int_{\R}\phi(y)\mu_s(dy)\leq \int_{\R}\phi(y)\mu_t(dy).
\end{equation}
If (\ref{eq:DcrStochOrder}) holds only for non-decreasing convex functions, then $\mu$ is said to be non-decreasing in {\it the increasing convex order}. If $\mu$ is non-decreasing in the usual stochastic order, resp. in the increasing convex order, then its image $\mu^h=\left(\mu_t^h,t\geq0\right)$ under $h:\,y\longmapsto -y$ is said to be non-decreasing in {\it the decreasing stochastic order}, resp. in {\it the decreasing convex order}. We also recall the definition of the MRL ordering which resembles that of the WDS ordering. Suppose that $\mu_t$ is integrable for every $t$. The family $\mu$ is said to be non-decreasing in  {\it the MRL order}  if the family of functions $\left(\Psi_{\mu_t}^{mrl},t\in\R_+\right)$ given by
\begin{equation*}
\Psi^{mrl}_{\mu_t}(x)=\left\{
\begin{array}{ll}
\dfrac{1}{\mu_t([x,+\infty[)}\displaystyle\int_{[x,+\infty[}y\mu_t(dy)&\text{if }x<r_{\mu_t},\\
x &\text{otherwise,}
\end{array}
\right.
\end{equation*}
where $r_{\mu_t}=\inf\{z\in\R:\,\mu_t([z,+\infty[)=0\}$, is pointwise non-decreasing. Now, we define the WDS ordering as follows. Suppose that, for every $t\geq0$, $\mu_t$ is integrable and has a negative mean.
The family $\mu$ is said to be non-decreasing in the weak decreasing stochastic (WDS) order if the  family of functions $(\Psi^{wds}_{\mu_t},t\geq0)$ defined by
\begin{equation*}
\Psi^{wds}_{\mu_t}(x)=\left\{
\begin{array}{ll}
\dfrac{1}{\mu_t([x,+\infty[)}\left(\displaystyle\int_{[x,+\infty[}y\mu_t(dy)-m_{\mu_t}\right)&\text{if }x<r_{\mu_t},\\
+\infty &\text{otherwise,}
\end{array}
\right.
\end{equation*}
where $m_{\mu_t}=\int_{\R} y\mu_t(dy)$, is pointwise non-decreasing. A family of integrable real-valued random variables with negative means is said to be non-decreasing in the WDS order if the family of their respective distributions is non-decreasing in the WDS order. Observe that the definition of the WDS ordering is the same as that of the MRL ordering up to the substraction of the mean of $\mu_t$. The terminology {\it weak decreasing stochastic ordering} has been chosen since, for processes with negative mean, the usual decreasing stochastic order is strictly stronger than the WDS order. Indeed, we show that every stochastically non-increasing process with negative mean is ordered by the WDS order and we exhibit some WDS ordered processes which do not decrease stochastically. On the other hand, we prove that, for processes with negative mean, the WDS order strictly implies the decreasing convex order. In particular, WDS ordered processes with constant negative mean are necessarily stochastically constant. Such a result has been proved by Shaked-Shanthikumar \cite[Theorem 1.A.8]{ShS} for stochastically non-increasing processes. One may also define a notion of {\it weak increasing stochastic (WIS) ordering }for processes with positive means. A family of integrable probability measures $\nu=(\nu_t,t\geq0)$ with positive means is said to be non-decreasing in {\it the WIS order} if the family of functions $\left(\Psi^{wis}_{\nu_t},t\geq0\right)$ given by
$$
\forall\,t\geq0,\quad \Psi^{wis}_{\nu_t}(x)=\left\{
\begin{array}{ll}
\dfrac{1}{\nu_t(]-\infty,x])} \left(m_{\nu_t}-\displaystyle\int_{]-\infty,x]}y\nu_t(dy)\right)&\text{if }x>l_{\nu_t}\\
+\infty&\text{otherwise,}
\end{array}
\right.
$$
where $m_{\nu_t}=\int y\nu_t(dy)$ and $l_{\nu_t}=\sup\{z\in\R:\,\nu_t(]-\infty,z])=0\}$. Observe that, if $\nu^h_t$ denotes the image of $\nu_t$ under $h$, then, for every $(t,x)\in\R_+\times\R$, $\Psi^{wis}_{\nu_t}(x)=\Psi^{wds}_{\nu^h_t}(-x)$. This implies that   $\nu=(\nu_t,t\geq0)$ is non-decreasing in the WIS order if, and only if   $\nu^h=\left(\nu^h_t,t\geq0\right)$ is non-decreasing in the WDS order. As a consequence, the WIS ordering is strictly weaker than the usual stochastic ordering and strictly stronger than the increasing convex ordering. Recently, Ewald and Yor \cite[Definition 1]{EY} introduced the notion of {\it a lyrebird} and called lyrebird a process that is non-decreasing in the increasing convex order. Hence the class of lyrebirds includes strictly that of WIS ordered processes.

There is a connection between the WDS ordering and the Cox-Hobson embedding. Indeed, the Cox-Hobson stopping time $T_{\mu_t}$ that solves the Skorokhod embedding problem\footnote{The Skorokhod embedding problem, which was first stated and solved by Skorokhod \cite{Sk}, may be described as follows:
Given a Brownian motion $\left(B_v,v\geq0\right)$ and a centered target law $\mu$, does there exist a
stopping time $T$ such that $B_T$ has distribution $\mu$?} for $\mu_t$ is the first time the process $\left(B_v,S_v:=\sup_{0\leq w\leq v}B_v;v\geq0\right)$ hits the epigraph $\mathcal{E}^{wds}_{\mu_t}$ of $\Psi_{\mu_t}^{wds}$, where $(B_v,v\geq0)$ denotes a standard Brownian motion issued from $0$. Hence the family $(\mu_t,t\geq0)$ is non-decreasing in the WDS order if, and only if $(\mu_t,t\geq0)$ can be embedded in a standard Brownian motion meaning that $t\longmapsto T_{\mu_t}$ is a.s. non-decreasing. Since  each $T_{\mu_t}$ is minimal (in a sense that is made more precised in the sequel),  the process $\left(B_{T_{\mu_t}},t\geq0\right)$ is a supermartingale with the same one-dimensional marginals as $(\mu_t,t\geq0)$. Then it follows from the Jensen inequality that the WDS order is stronger than the decreasing convex order. We recover this property using a different approach. Note that $\left(B_{T_{\mu_t}},t\geq0\right)$ is Markovian when the distributions $\mu_t$, $t\in\R_+$ have densities. This follows from a similar argument than that used in Madan-Yor \cite[Theorem 2]{MY}. If there is some distribution $\mu_s$ with atoms, then the atomic part of $\mu_s$ makes some parts of $\mathcal{E}^{wds}_{\mu_s}$ vertical. Thus, for $s<t$, the future random time $T_{\mu_t}$ does not only depend of $B_{T_{\mu_s}}$ but also on $S_{T_{\mu_s}}$. Then $\left(B_{T_{\mu_t}},t\geq0\right)$ is not Markovian.
We say that two processes are associated if they have the same one-dimensional marginals. Let us mention that the problem of existence of a  supermartingale associated to a given process which is ordered by the  decreasing convex order was solved by Kellerer \cite{Kel}.

In this paper, we provide a log-concavity characterization of the WDS ordering. This characterization is the same as that obtained in Bogso \cite[Theorem 3.3]{Bo} for the MRL ordering.  

We organize the rest of the paper as follows. In the next Section, we briefly recall the Cox-Hobson \cite{CH} extension of the Az\'ema-Yor algorithm to target distributions with negative mean and, using a Madan-Yor argument, deduce that WSD ordering is a necessary and sufficient condition for an integrable process with negative mean to embeddable in Brownian motion by the generalized Azéma-Yor stopping times. Section 3 is devoted to  a characterization of the WDS ordering in terms of log-concavity and to some of its closure properties. Finally, in Section 4, we present several examples of  WDS ordered processes.

\section{WDS order and the Cox-Hobson algorithm}
\label{sec:1}
Let $(\mu_t,t\geq0)$ be a family of integrable probability measures with negative mean. For every $t\geq0$, we set $m_{\mu_t}:=\int y\mu_t(dy)$.   We shall apply the Cox-Hobson \cite{CH} extension of the Azéma-Yor algorithm to embed simultaneously all distributions $\mu_t$ in a standard Brownian motion $(B_t,t\geq0)$ started at $0$. In the zero-mean case, if $(\mu_t,t\geq0)$ is MRL ordered, then the original Azéma-Yor algorithm \cite{AY} provides a family of a.s. non-decreasing stopping times $\left(T^{AY}_{\mu_t},t\geq0\right)$ adapted to the natural filtration of $(B_t,t\geq0)$ such that, for every $t\geq0$:
\begin{enumerate}
\item[C1)] $\left(B_{T^{AY}_{\mu_t}\wedge v},v\geq0\right)$ is uniformly integrable,
\item[C2)]the law of $B_{T^{AY}_{\mu_t}}$ is $\mu_t$.
\end{enumerate}
Since Condition C1) is equivalent to $\E\left[B_{T^{AY}_{\mu_t}}|\mathcal{F}_{S}\right]=B_{S}$ for all stopping times $S\leq T^{AY}_{\mu_t}$, then $\left(B_{T^{AY}_{\mu_t}},t\geq0\right)$ is a martingale with marginals $(\mu_t,t\geq0)$. Moreover, Monroe \cite{Mo} showed that Condition C1) is equivalent to the statement that $T_{\mu_t}$ is {\it minimal} in the following sense.
\begin{definition}\label{def:minimality}
A stopping time $T$ is said to be {\it minimal } for $(B_t,t\geq0)$ if whenever $S\leq T$ is a stopping time such that $B_{T}$ and $B_{S}$ have the same distribution then $S=T$ a.s..
\end{definition}
The result of Monroe has been extended by Cox and Hobson \cite{CH} to non-centered target distribution. In particular, provided   $m_{\mu_t}<0$, the authors proved that a stopping time $T$ is minimal for $(B_t,t\geq0)$ if, and only if  $\left(B^{-}_{T\wedge v},v\geq0\right)$ is uniformly integrable, or  equivalently if $\E[B_T|\mathcal{F}_S]\leq B_S$ for all stopping times $S\leq T$.  Cox and Hobson \cite{CH} also constructed an extension of the Azéma-Yor embedding for distributions with negative means and showed that it is minimal. Moreover, the Cox-Hobson stopping time  retains the optimality properties of the original Azéma-Yor embedding. Precisely, the Cox-Hobson stopping time, denoted by $T^{CH}_{\mu_t}$, maximises the law of $\sup_{v\leq R}B_v$ for the usual stochastic order among all minimal stopping times $R$ which embed $\mu_t$. We refer the reader to the remarkable paper of Beiglböck-Cox-Huesmann \cite{BCH} where a systematic method to construct optimal skorokhod embeddings is provided. Here is a brief exposition of the Cox-Hobson algorithm for a fixed $t\geq0$. 
Consider the convex function 
\[
\forall\,x\in\R,\text{ }\pi_{\mu_t}(x)=\int_{\R}|y-x|d\mu_t(y)-m_{\mu_t}.
\]
and, for every $\theta\in[-1,1]$, define 
\[
u_{\mu_t}(\theta)=\inf\{y\in\R:\,\pi_{\mu_t}(y)+\theta(x-y)\leq\pi_{\mu_t}(x), \text{for every }x\in\R\}.
\]
Note that $\pi_{\mu_t}$ is asymptotic to and greater than the function 
\[
x\longmapsto-x1_{]-\infty,m_{\mu_t}]}(x)+(x-2m_{\mu_t})1_{]m_{\mu_t},+\infty[}(x).
\]
On the other hand, if $\theta$ is interpreted as the gradient of a tangent to $\pi_{\mu_t}$, then $u_{\mu_t}(\theta)$ is the smallest $x$ at which there exists a tangent to $\pi_{\mu_t}$ with gradient $\theta$.
Define also 
\[
\forall\,\theta\in[-1,1],\text{ }z_{\mu_t}(\theta)=\frac{\pi_{\mu_t}(u_{\mu_t}(\theta))-\theta u_{\mu_t}(\theta)}{1-\theta}
\]
and
\[
\forall\,\alpha\in\R_+,\text{ }b_{\mu_t}(\alpha)=u_{\mu_t}\left(z_{\mu_t}^{-1}(\alpha)\right),
\]
where
\[
z_{\mu_t}^{-1}(\alpha)=\inf\{\theta\in[-1,1]:\,z_{\mu_t}(\theta)\geq\alpha\}.
\]
These definitions are interpreted as follows. Consider the unique tangent to $\pi_{\mu_t}$ with gradient $\theta$. Then, $z_{\mu_t}(\theta)$ is the $x$-coordinate of the point where this tangent   crosses the line $y=x$ and $b_{\mu_t}(\alpha)$ is the  $x$-coordinate of the  left-most point of the graph of $\pi_{\mu_t}$ for which the tangent to $\pi_{\mu_t}$ at that point hits the line $y=x$ at $\alpha$. We refer to Figures 1  for a pictorial representation of $b_{\mu_t}$ in the case $m_{\mu_t}<0$.

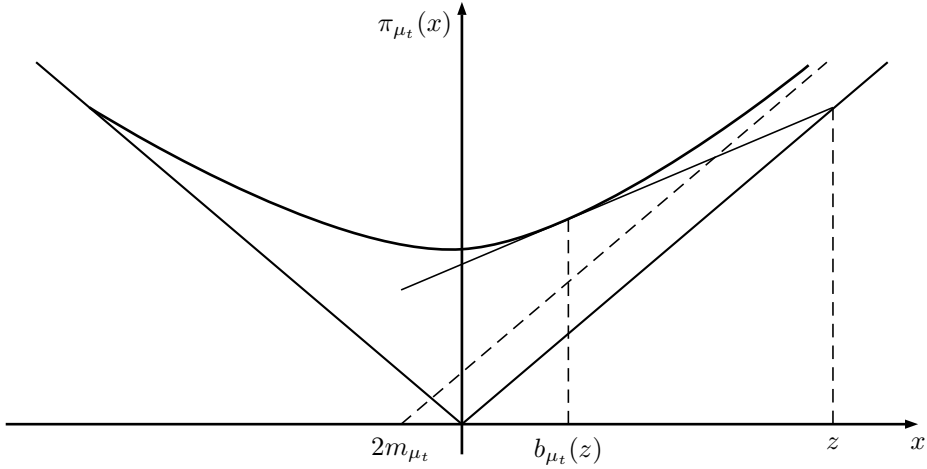
\begin{figure}


\ifx\JPicScale\undefined\def\JPicScale{.5}\fi
\psset{xunit=.4cm,yunit=.4cm}
\psset{dotsize=0.7 2.5,dotscale=1 1,fillcolor=black}
\begin{pspicture*}(-16,-2)(16,15) 
\psline[linewidth=1pt,arrowinset=0]{->}(-15,0)(15,0) 
\psline[linewidth=1pt,arrowinset=0]{->}(0,-1)(0,14)
\psline[linewidth=0.8pt](0,0)(14,12)
\psline[linewidth=0.8pt](0,0)(-14,12)
\psline[linewidth=0.6pt, linestyle=dashed](-2,0)(12,12)  
\pscurve[linewidth=1pt](-12.25,10.5)(0,5.8) (11.4,11.9)
\psline[linewidth=0.6pt](-2,4.45)(12.2,10.5)
\psline[linewidth=0.6pt, linestyle=dashed](12.2,0)(12.2,10.5)
\psline[linewidth=0.6pt, linestyle=dashed](3.5,0)(3.5,6.8)
\uput[d](12.2,0){$z$}
\uput[d](3.5,0){$b_{\mu_t}(z)$}
\uput[d](-2,0){$2m_{\mu_t}$}
\uput[d](15,0){$x$}
\uput[dl](0,14){$\pi_{\mu_t}(x)$}
\end{pspicture*}
\caption{$\pi_{\mu_t}$ for a $\mu_t$ with support bounded below and negative mean.}
\label{fig:1}       
\end{figure}

The following result has been proved by Cox-Hobson \cite{CH}.
\begin{theorem}\label{theo:CHmin}(Cox-Hobson \cite[Theorem 12]{CH}).
The stopping time $T_{\mu_t}^{CH}$ given by
\[
T_{\mu_t}^{CH}=\inf\{v\geq0:\, b_{\mu_t}(S_v)\geq B_v\},
\]
where $S_v=\sup\limits_{w\leq v}B_w$, embeds $\mu_t$ and is minimal.  
\end{theorem}
 
\begin{remark}\text{}
\item[1)] Let $b_{\mu_t}^{-1}$ be the function defined on $\R$ by
\[
b_{\mu_t}^{-1}(x)=\inf\{z\geq0:\,b_{\mu_t}(z)\geq x\}.
\]
Then, as in Cox \cite[Chapter 3, Section 3.2]{Co}, one may show that  $b_{\mu_t}^{-1}=\Psi^{wds}_{\mu_t}$, which implies that 
\begin{equation*}
T_{\mu_t}^{CH}=\inf\{v\geq0:\,S_v\geq\Psi^{wds}_{\mu_t}(B_v)\}.
\end{equation*} 
Observe that $T_{\mu_t}^{CH}$ is the first time the process $(B_v,S_v;v\geq0)$ hits a barrier, that is the epigraph of $\Psi_{\mu_t}^{wds}$.  
\item[2)]We have $\lim\limits_{x\to-\infty}\Psi_{\mu_t}^{wds}(x)=0$. Moreover,  $\Psi^{wds}_{\mu_t}$ is left-continuous, non-decreasing and non-negative on $]-\infty,r_{\mu_t}[$. 
\end{remark}
 
Now, suppose that, for every $t\geq0$, $m_{\mu_t}<0$ and that the family of stopping times  $(T_{\mu_t}^{CH},t\geq0)$ is a.s. non-decreasing. Since $T^{CH}_{\mu_t}$ is minimal, then, Cox and Hobson \cite[Theorem 3]{CH} prove that $\left(B^{-}_{T_{\mu_t}^{CH}\wedge v},v\geq0\right)$ is uniformly integrable which is equivalent to $\E\left[B_{T^{CH}_{\mu_t}}|\mathcal{F}_S\right]\leq B_S$ for all stopping times $S\leq T^{CH}_{\mu_t}$. In particular, $\E\left[B_{T^{CH}_{\mu_t}}|\mathcal{F}_{T^{CH}_{\mu_s}}\right]\leq B_{T^{CH}_{\mu_s}}$ for every $s\leq t$ which means that  $\left(B_{T^{CH}_{\mu_t}},t\geq0\right)$ is a   supermartingale with marginals $(\mu_t,t\geq0)$. But, the family   $\left(T_{\mu_t}^{CH},t\geq0\right)$ is a.s. non-decreasing if, and only if  $(\Psi^{wds}_{\mu_t},t\geq0)$ is pointwise non-decreasing which means that the family $(\mu_t,t\geq0)$ is WDS ordered.  
 
The interest of the preceeding results rely on the existence of numerous  WDS processes. We present many of them in Section 4. Before doing that, we  provide  an equivalent log-concavity property to WDS ordering and give some closure results related to this ordering.

\section{Log-concavity properties of the WDS ordering }
We give a characterization of the WDS ordering in terms of total positivity of order 2. We then deduce several preservation properties which are helpful to construct other families of WDS processes. Let us start with some general facts about totally positive of order 2 (TP$_2$) functions.
\subsection{Definition and properties of TP$_2$ functions}
We define and give some examples of TP$_2$ functions. We also present a composition result  which allows to generate other TP$_2$ functions.
\begin{definition}
Let $I$ and $I'$ be two intervals of $\R$. Let $L$ be a real function defined on $I\times I'$. The function $L$ is said to be totally positive of order 2 (TP$_2$) on $I\times I'$ if, for every $x_1\leq x_2$ elements of $I$ and every $y_1\leq y_2$ elements of $I'$,
\begin{equation}\label{eq:defiTP2funct}
\det\left(
\begin{array}{cc}
L(x_1,y_1)&L(x_1,y_2)\\
L(x_2,y_1)&L(x_2,y_2)
\end{array}
\right)\geq0.
\end{equation}
\end{definition}
\noindent
There are many examples of TP$_2$ functions in the literature among which there are TP$_2$ transition densities and TP$_2$ integrated survival functions.
\begin{example}\label{exa:ExaTP2Funct}\text{}
\item[1.]Let $f$ be a density function defined on $\R$. Then $f$ is log-concave if, and only if the function $(x,y)\longmapsto f(x-y)$ is TP$_2$ on $\R\times\R$, i.e. for every real numbers $x_1\leq x_2$ and $y_1\leq y_2$,
\begin{equation}\label{eq:defiPF2Funct}
\det\left(
\begin{array}{cc}
f(x_1-y_1)&f(x_1-y_2)\\
f(x_2-y_1)&f(x_2-y_2)
\end{array}
\right)\geq0.
\end{equation}
\item[2.]Let $(\Lambda_t,t\geq0)$ be a one-dimensional diffusion whose transition kernel is of the form $P_t(x,dy)=p_t(x,y)dy$, where $p_t$ is continuous. Then $p_t$ is TP$_2$ on $\R\times \R$ (see e.g. Karlin-Taylor \cite[Chap. 15, Problem 21]{KaT}). If we suppose in addition that $(\Lambda_t,t\geq0)$ is $\R_+$-valued and starts at $0$, then, as stated in Karlin \cite[Theorem 5.2]{KA}, the function $(t,x)\longmapsto p_t(0,x)$ is TP$_2$ on $\R_+^{\ast}\times \R_+$.
\item[3.]Let $(X_t,t\geq0)$ be an integrable real-valued process. Then $(X_t,t\geq0)$ is non-decreasing in the MRL order if, and only if the integrated survival function $(t,x)\longmapsto\E\left[(X_t-x)^+\right]$ is TP$_2$ on $\R_+\times\R$ (see Bogso \cite[Theorem 3.3]{Bo}).
\end{example} 
\noindent
Other examples of TP$_2$ functions may be found in \cite[Sections 2-5]{KA}  and \cite[Section 2.2]{Bo1}. The following result follows directly from the Cauchy-Binet formula (see e.g. \cite[Paragraph 1.1]{Ma} or \cite[Problem 11.1.28]{PS}) and states that the TP$_2$ property is preserved by composition.
\begin{prop}\label{prop:Cauchy-BinetTP2}
Let $I$, $I'$ and $J$ be sub-intervals of $\R$, and let $L:\,I\times J\to\R_+$ and $M:\,J\times I'\to\R_+$ be two TP$_2$ functions. Then, for every $\sigma$-finite measure $\eta$ such that 
$$
\int_JL(x,z)M(z,y)\,\eta(dz)\quad\text{is finite,}
$$
the function $N$ defined on $I\times I'$ by $N(x,y)=\displaystyle\int_JL(x,z)M(z,y)\,\eta(dz)$ is TP$_2$.
\end{prop}
\noindent
Here is another transformation which preserves the property of being TP$_2$.
\begin{prop}\label{prop:InterPolTP2}
Let $I$ and $I'$ be two sub-intervals of $\R$, and let $L:\,I\times I'\to\R_+$ be a TP$_2$ function. For every positive integer $r$ and every increasing sequence $\mathbf{a}=(a_0,a_1,\cdots,a_r)$ in $I$, the function $L^{\mathbf{a}}$ defined on $I\times I'$ by: For every $y\in I'$,
\begin{equation}\label{eq:defiCensoringGle}
L^{\mathbf{a}}(x,y)=\left\{
\begin{array}{cc}
L(x,y)&\text{if }x\notin[a_0,a_r] \\
\dfrac{a_1- x}{a_1-a_0}L(a_0,y)+\dfrac{x-a_0}{a_1-a_0}L(a_{1},y)&\text{if }x\in[a_0,a_{1}]\\
\cdots\cdots\cdots\cdots  & \cdots\cdots  \\
\dfrac{a_r- x}{a_r- a_{r-1}}L(a_{r-1},y)+\dfrac{x- a_{r-1}}{a_r- a_{r-1}}L(a_{r},y)&\text{if }x\in[a_{r-1},a_{r}]
\end{array}
\right.
\end{equation}
is TP$_2$.
\end{prop}
\begin{proof}
We first treat the case $r=1$. We wish to show that the function $L^{(a_0,a_1)}$ defined on $I\times I'$ by: for every $y\in I'$,
\begin{equation}\label{eq:defi01Censoring}
L^{(a_0,a_1)}(x,y)=\left\{
\begin{array}{cl}
L(x,y)&\text{if }x\notin[a_0,a_1] \\
\dfrac{a_1-x}{a_1-a_0}L(a_0,y)+\dfrac{x-a_0}{a_1-a_0}L(a_{1},y)&\text{otherwise}
\end{array}
\right.
\end{equation}
is TP$_2$. \\ Let $x_1\leq x_2$ be elements of $I$ and let $y_1\leq y_2$ be elements of $I'$. We distinguish four cases.
\begin{enumerate}
\item[(i)]If $x_1\notin[a_0,a_1]$ and $x_2\notin[a_0,a_1]$, then
\begin{align*}
L^{(a_0,a_1)}\left(
\begin{array}{cc}
x_1&x_2\\
y_1&y_2
\end{array}
\right)&:=L^{(a_0,a_1)}(x_1,y_1)L^{(a_0,a_1)}(x_2,y_2)-L^{(a_0,a_1)}(x_1,y_2)L^{(a_0,a_1)}(x_2,y_1)\\
&=L(x_1,y_1)L(x_2,y_2)-L(x_1,y_2)L(x_2,y_1)\\&:=L\left(
\begin{array}{cc}
x_1&x_2\\
y_1&y_2
\end{array}
\right)\geq0\quad(\text{since }L\text{ is TP}_2).
\end{align*}
\item[(ii)]If $x_1<a_0\leq x_2\leq a_1$, then
\begin{align*}
 L^{(a_0,a_1)}\left(
\begin{array}{cc}
x_1&x_2\\
y_1&y_2
\end{array}
\right) 
 =\dfrac{x_2-a_0}{a_1-a_0}L\left(
\begin{array}{cc}
x_1&a_1\\
y_1&y_2
\end{array}
\right)+\dfrac{a_1-x_2}{a_1-a_0}L\left(
\begin{array}{cc}
x_1&a_0\\
y_1&y_2
\end{array}
\right)\geq0.
\end{align*}
\item[(iii)]If $a_0\leq x_1\leq a_1<x_2$, then
\begin{align*}
L^{(a_0,a_1)}\left(
\begin{array}{cc}
x_1&x_2\\
y_1&y_2
\end{array}
\right) 
=\dfrac{x_1-a_0}{a_1-a_0}L\left(
\begin{array}{cc}
a_1&x_2\\
y_1&y_2
\end{array}
\right)+\dfrac{a_1-x_1}{a_1-a_0}L\left(
\begin{array}{cc}
a_0&x_2\\
y_1&y_2
\end{array}
\right)\geq0.
\end{align*}
\item[(iv)]If $a_0\leq x_1\leq x_2\leq a_1$, then
$$
L^{(a_0,a_1)}\left(
\begin{array}{cc}
x_1&x_2\\
y_1&y_2
\end{array}
\right)=\dfrac{x_2-x_1}{a_1-a_0}L\left(
\begin{array}{cc}
a_0&a_1\\
y_1&y_2
\end{array}
\right)\geq0.
$$
\end{enumerate}
We deduce from (i)-(iv) that $L^{(a_0,a_1)}$ is TP$_2$. Moreover, if $\mathcal{T}^{\mathbf{a}}$ and $\mathcal{T}^{a_0,a_1}$ denote the transformations defined by (\ref{eq:defiCensoringGle}) and (\ref{eq:defi01Censoring}) respectively, then $\mathcal{T}^{\mathbf{a}}$ equals the $r$-fold composition $\mathcal{T}^{a_0,a_1}\circ\mathcal{T}^{a_1,a_2}\circ\cdots\circ \mathcal{T}^{a_{r-1},a_r}$. As a consequence, $L^{\mathbf{a}}$ is also TP$_2$.
\end{proof}  
\subsection{A log-concavity characterization of the WDS ordering}
For a family $\mu=(\mu_t,t\geq0)$ of integrable probability measures, we consider its integrable survival function $C_\mu$ given by:
\[
\forall\,(t,x)\in\R_+\times\R,\,C_{\mu}(t,x)=\int_{[x,+\infty[}(y-x)\mu_t(dy)
\]
and the function $K_\mu$ defined by
\begin{equation}\label{eq:starMRLfunct}
\forall\,(t,x)\in\R_+\times\R,\,K_{\mu}(t,x)=C_{\mu}(t,x)-m_{\mu_t},
\end{equation}
where $m_{\mu_t}$ denotes the mean of $\mu_t$. As stated in the next result, borrowed from  Hirsch-Roynette \cite[Section 2]{HR} (see also Müller-Stoyan \cite[Theorem 1.5.10]{MS}), the function $C_{\mu}$ determines entirely the family $(\mu_t,t\geq0)$. 
\begin{prop}\label{prop:DefCallFunct}
Let $\nu$ be an integrable probability measure and let $C_{\nu}$ denote the integrated survival function of $\nu$, i.e.  $C_{\nu}(x)=\displaystyle\int_{[x,+\infty[}(y-x)\nu(dy)$ for every $x\in\R$. 
Then $C_{\nu}$ enjoys the following properties:
\begin{enumerate}
\item[i)]$C_{\nu}$ is a convex, nonnegative function on $\R$,
\item[ii)]$\lim\limits_{x\to+\infty}C_{\nu}(x)=0$,
\item[iii)]there exists $l\in\R$ such that $\lim\limits_{x\to-\infty}(C_{\nu}(x)+x)=l$.  
\end{enumerate}
Conversely, if a function $C$ satisfies the above three properties, then there exists a unique integrable probability measure $\nu$ such that $C_{\nu}=C$, i.e. $C$ is the integrated survival function of $\nu$.
Precisely, $\nu$ is the second order derivative of $C$ in the sense of distributions, and $l=\int_{\R}y\nu(dy)$. 
\end{prop}
 
Here is a characterization of WDS ordering in terms of log-concavity for  probability measures with negative mean. This characterization is the same as that of the MRL ordering obtained in Bogso \cite[Theorem 3.3]{Bo}.
\begin{theorem}\label{theo:WDSTPTail}
Suppose that, for every $t\geq0$, $m_{\mu_t}:=\int_{\R}y\mu_t(dy)<0$. Then, the family $(\mu_t,t\geq0)$ is non-decreasing in the WDS order if and only if the function $K_{\mu}$ defined by (\ref{eq:starMRLfunct}) is TP$_2$,
i.e. for every $0\leq t_1\leq t_2$ and $x_1\leq x_2$,
\begin{equation}\label{eq:starMRLTailTP2}
\det
\begin{pmatrix}
K_{\mu}(t_1,x_1)&K_{\mu}(t_1,x_2)\\
K_{\mu}(t_2,x_1)&K_{\mu}(t_2,x_2)
\end{pmatrix}
\geq0.
\end{equation}
\end{theorem}
\begin{proof}
We first assume that the family $(\Psi^{wds}_{\mu_v},v\geq0)$ is pointwise non-decreasing. Let $0\leq s\leq t$ and $x\leq y$ be fixed. We wish to prove that 
\begin{equation}\label{eq:starMRLpr0}
K_\mu(s,x)K_\mu(t,y)\geq K_\mu(s,y)K_\mu(t,x).
\end{equation}
We start by observing that $r_{\mu_t}\leq r_{\mu_s}$. Indeed, by hypothesis and by definition of $\Psi^{wds}_{\mu_s}$, we have $\Psi^{wds}_{\mu_t}(r_{\mu_s})\geq \Psi^{wds}_{\mu_s}(r_{\mu_s})=+\infty$. Moreover, $x\geq r_{\mu_t}$ is equivalent to $C_{\mu}(t,x)=0$, which in turn is equivalent to $K_{\mu}(t,x)=-m_{\mu_t}$. Now, since $z\longmapsto K_\mu(s,z)$ and $z\longmapsto K_\mu(t,z)$ are continuous and positive, then, to obtain (\ref{eq:starMRLpr0}), it suffices to show that
\begin{equation}\label{eq:starMRLpr12}
G_{s,t}:z\longmapsto\frac{K_\mu(t,z)}{K_\mu(s,z)}\text{ is non-decreasing on both }]-\infty,r_{\mu_t}[\text{ and }[r_{\mu_t},+\infty[.
\end{equation}
Indeed, suppose that (\ref{eq:starMRLpr12}) holds; if $x$ and $y$ are taken such that  $x<r_{\mu_t}\leq y$, then, denoting by $G_{s,t}\left(r^-_{\mu_t}\right)$ the left-hand limit of $G_{s,t}$ at $r_{\mu_t}$, we have $$G_{s,t}(x)\leq G_{s,t}\left(r_{\mu_t}^-\right)=G_{s,t}(r_{\mu_t})\leq G_{s,t}(y).$$
Now, since $z\longmapsto K_{\mu}(s,z)$ is non-increasing, $G_{s,t}$ is non-decreasing on $[r_{\mu_t},+\infty[$. Moreover, $G_{s,t}$ is left-differentiable and its left-derivative  $G'_{s,t}$ satisfies:
\[
\forall\,z<r_{\mu_t},\text{ }K^2_\mu(s,z)G'_{s,t}(z)=\mu_s([z,+\infty[)\mu_t([z,+\infty[)\left(\Psi^{wds}_{\mu_t}(z)-\Psi^{wds}_{\mu_s}(z)\right)\geq0,
\]
where $\mu_s([z,+\infty[)>0$ and $\mu_t([z,+\infty[)>0$ since $z<r_{\mu_t}\leq r_{\mu_s}$. Hence, as $G_{s,t}$ is continuous, we deduce that $G_{s,t}$ is still non-decreasing on $]-\infty,r_{\mu_t}[$.

Conversely, suppose that $K_\mu$ is TP$_2$ on $\R_+\times\R$. Let $0\leq s\leq t$ and $x\in\R$ be given. We shall prove that 
\begin{equation}\label{eq:starMRLpr3}
\Psi^{wds}_{\mu_s}(x)\leq\Psi^{wds}_{\mu_t}(x).
\end{equation}
If $x\geq r_{\mu_t}$, then (\ref{eq:starMRLpr3}) is immediate. Now, suppose that $x<r_{\mu_t}$. Since $K_\mu$ is TP$_2$, then (\ref{eq:starMRLpr0}) holds in particular for $y\geq r_{\mu_t}$ which implies that
\[
K_\mu(s,x)\geq-\frac{K_\mu(t,x)}{m_{\mu_t}}K_\mu(s,y)>K_\mu(s,y)\geq-m_{\mu_s}
\] 
which in turn implies that $x<r_{\mu_s}$. The TP$_2$ property of $K_\mu$ implies also that the function $G_{s,t}$ is non-decreasing on $]-\infty,x]$. In particular, the left-derivative $G'_{s,t}$ of $G_{s,t}$ at $x$ is nonnegative. Therefore, we have:
\[
0\leq K^2_\mu(s,x)G'_{s,t}(x)=\mu_s([x,+\infty[)\mu_t([x,+\infty[)\left(\Psi^{wds}_{\mu_t}(x)-\Psi^{wds}_{\mu_s}(x)\right)
\]
which ends the proof.
\end{proof}

It is proved in Shaked-Shanthikumar \cite[Theorem 4.A.26]{ShS} that the MRL ordering implies the increasing convex ordering. In the next result, we show that the WDS ordering implies the decreasing convex ordering.
\begin{theorem}\label{theo:starMRLIncv}
Every WDS ordered process $(X_t,t\geq0)$ with negative mean is ordered by the decreasing convex order. 
\end{theorem}
\begin{proof}
Let $\mu:=(\mu_t,t\geq0)$ denote the family of one-dimensional marginals of $(X_t,t\geq0)$. Let $x\in\R$ and $0\leq s<t$ be fixed. Since, for every $v\geq0$, $K_\mu(v,x)+x=\E\left[(x-X_v)^+\right]$, then it suffices to prove that
\[
K_\mu(s,x)\leq K_\mu(t,x).
\]
We deduce from Point iii) of Proposition \ref{prop:DefCallFunct} that
\begin{equation}\label{eq:CallFunctAsympt}
\lim\limits_{z\to-\infty}\frac{K_\mu(s,z)}{z}=-1=\lim\limits_{z\to-\infty}\frac{K_\mu(t,z)}{z}.
\end{equation}
On the other hand, the TP$_2$ property of $K_\mu$ yields that
\begin{equation}\label{eq:starCallTP2}
\forall\,z\leq x,\text{ }\frac{K_\mu(t,z)}{K_\mu(s,z)}\leq\frac{K_\mu(t,x)}{K_\mu(s,x)}.
\end{equation}
As a consequence of (\ref{eq:CallFunctAsympt}) and (\ref{eq:starCallTP2}), we have:
\[
1=\lim\limits_{z\to-\infty}\frac{K_\mu(t,z)}{K_\mu(s,z)}\leq\frac{K_\mu(t,x)}{K_\mu(s,x)}
\]
which completes the proof.
\end{proof}
\begin{remark}
It follows from Theorem \ref{theo:starMRLIncv} and from the Kellerer's theorem that every  WDS ordered process with negative mean has the same one-dimensional marginals as a Markovian supermartingale. 
\end{remark}
Now, we wish to prove that the implication of Theorem \ref{theo:starMRLIncv} is strict. For this end, we should provide a process which increases in the decreasing convex order and which is not WDS ordered. The following result states that non trivial peacock processes do not increase in the WDS order. Precisely, WDS ordered processes with constant mean are stochastically constant.  
\begin{theorem}\label{theo:wdspcoc}
Let $(X_t,t\geq0)$ be an integrable WDS process such that, for every $t\geq0$, $X_t$ has a negative mean. If $\E[X_t]$ does not depend on $t$, then, for every $0\leq s\leq t$, $X_s$ and $X_t$ have the same distribution.
\end{theorem}
\begin{proof}
Let $\mu_t$ denote the law of $X_t$. Fix $0\leq s\leq t$ and $x\in\R$. Then combining Point ii) of Proposition \ref{prop:DefCallFunct}, Theorem \ref{theo:WDSTPTail} and Theorem \ref{theo:starMRLIncv}, we obtain
$$
1=\lim\limits_{y\to-\infty}\frac{K_\mu(t,y)}{K_\mu(s,y)}\leq\frac{K_\mu(t,x)}{K_\mu(s,x)}\leq\lim\limits_{z\to+\infty}\frac{K_\mu(t,z)}{K_\mu(s,z)}=1
$$
which yields the desired result.
\end{proof}
\begin{remark}
The preceding property of the WDS ordering has already been proved for the usual stochastic ordering (see e.g. Shaked-Shanthikumar \cite[Theorem 1.A.8]{ShS}).
\end{remark}
Here are some closure properties of the WSD ordering which are useful to exhibit several other WDS ordered processes.

\subsection{Preservation properties of the WDS ordering}
The preservation properties listed below are satisfied by the MRL ordering. They follow directly from well-known facts related to total positivity of order 2. 
\subsubsection{Subordination}
Let $(Y_{\lambda},\lambda\geq0)$ be a WDS process. Let $(\Lambda_t,t\geq0)$ be an homogeneous and right-continuous $\R_+$-valued Markov process independent of $(Y_\lambda,\lambda\geq0)$. We suppose that $\Lambda_t$ has density $p_t$ and that the function $(t,\lambda)\longmapsto p_t(\lambda)$ is TP$_2$ on $\R_+^{\ast}\times\R_+$, where $\R_+^{\ast}$ stands for the set of positive real numbers. Some examples of such processes are given in Karlin \cite[Theorems 4.3(i) and 5.2]{KA}. We suppose that, for every $t\geq0$, $\E\left[\left|Y_{\Lambda_t}\right|\right]<\infty$, and we consider the process $\left(X_t:=Y_{\Lambda_t},t\geq0\right)$. Let $K_X$ and $K_Y$ denote the fuctions defined by $K_X(t,x)=\E\left[(X_t-x)^+\right]-\E[X_t]$ and $K_Y(\lambda,x)=\E\left[(Y_\lambda-x)^+\right]-\E[Y_\lambda]$ respectively.
We have
$$
K_X(t,x)=\int_{\R_+}K_Y(\lambda,x)p_t(\lambda)\,d\lambda
$$
which, by Proposition \ref{prop:Cauchy-BinetTP2}, implies that  $\left(X_t:=Y_{\Lambda_t},t\geq0\right)$ is WDS ordered.

\subsubsection{Random translation}\label{sssec:RandTranslate}
Let $(X_t,t\geq0)$ be an integrable WDS process with negative mean and let  
$Y$ be an integrable centered log-concave random variable independent of $(X_t,t\geq0)$.
Then $(Z_t:=X_t+Y,t\geq0)$ is still a WDS ordered process. Indeed,
 if we define the functions $K_X$, resp. $K_Z$ by $K_X(t,x)=C_X(t,x)-m_t$, resp. $K_Z(t,x)=C_Z(t,x)-m_t$, where $C_X(t,\cdot)$, resp. $C_Z(t,\cdot)$ denotes the integrated survival function of $X_t$, resp. $Z_t$ and where $m_t$ is the mean of $X_t$  and if we denote by  $f_Y$ the density of $Y$, then
$$
K_Z(t,x)=\int_{\R}K_X(t,x-y)f_Y(y)dy=\int_{\R}K_X(t,z)f_Y(x-z)dz.
$$
It follows from Point 1. of Example \ref{exa:ExaTP2Funct} and from Proposition \ref{prop:Cauchy-BinetTP2}, that $K_Z$ is TP$_2$. 

Note that the WDS ordering is not stable by deterministic translations: let $(\mu_t,t>0)$ be the family of distributions given by  
$$
\forall\,t>0,\quad\mu_t=\frac{1}{1+t}\delta_{-t}+\frac{t}{1+t}\delta_{1-t},
$$
where $\delta_{-t}$ and $\delta_{1-t}$ are the Dirac measures at points $-t$ and $1-t$ respectively. Then $(\mu_t,t>0)$ is non-decreasing in the WDS order. On the contrary, denoting by $\widehat{\mu}_t$ the image of $\mu_t$ under $y\longmapsto y-1$, the process $\left(\widehat{\mu}_t,t>0\right)$ is not WDS ordered.
\subsubsection{Scale mixtures}
Let  $(X_t,t\geq0)$ be an integrable WDS ordered process with negative mean and let
 $Y$ be an integrable $\R_+$-valued random variable independent of $(X_t,t\geq0)$ which admits a positive $\mathcal{C}^1$-class density $f$. We suppose that $\log Y$ is log-concave. We consider the the process  $(Z_t:=YX_t,t\geq0)$. Let $K_Z$ and $K_X$ be defined as in Paragraph 3.3.2. One has:
$$
\forall\,(t,x)\in\R_+\times\R,\quad K_Z(t,x)=\int_0^{\infty}K_X\left(t,\frac{x}{y}\right)yf(y)dy,
$$
and, making the change of variable $zy=x$ for a fixed $x\in\R^\ast$, 
$$
K_Z(t,x)=\left\{
\begin{array}{ll}
\displaystyle\int_{0}^{\infty}K_X(t,z)f\left(\dfrac{x}{z}\right)  \dfrac{x^2dz}{z^3} &\text{if }x>0,\\
-\displaystyle\int_{-\infty}^{0}K_X(t,z)f\left(\dfrac{x}{z}\right)  \dfrac{x^2dz}{z^3}&\text{if }x<0.
\end{array}
\right.
$$
Since $\log Y$ is log-concave, $(x,z)\longmapsto f(x/z)$ is TP$_2$ on $\R_+^\ast\times\R_+^\ast$ and, as a consequence, TP$_2$ on  $\R_-^\ast\times\R_-^\ast$, where $\R_-^\ast$ denotes the set of negative real numbers. Then, by Proposition \ref{prop:Cauchy-BinetTP2}, $K_Z$ is TP$_2$ on $\R_+\times\R_+^\ast$  and on $\R_+\times\R_-^\ast$. Since $K_Z$ is continuous, we deduce that $K_Z$ is still TP$_2$ on $\R_+\times\R$.
Then $(Z_t:=YX_t,t\geq0)$ is a WDS ordered process.  
 
\subsubsection{Convex combinations}
Let $r$ be a positive integer and let $(\mu_t,t\geq0)$ be a WDS ordered family of probability measures. Let $\tau=(\tau_0,\tau_1,\cdots,\tau_r)$ be an increasing sequence of nonnegative real numbers. Consider the family $(\mu^{(\tau)}_t,t\in\R_+)$ given by: 
$$
\mu^{(\tau)}_t=\left\{
\begin{array}{cc}
\mu_t&\text{if }t\notin[\tau_0,\tau_r]\\
\dfrac{\tau_{1}-t}{\tau_{1}-\tau_0}\mu_{\tau_0}+\dfrac{t-\tau_0}{\tau_{1}-\tau_{0}}\mu_{\tau_1}&\text{if }t\in[\tau_0,\tau_1]\\
\cdots\cdots\cdots&\cdots\\
\dfrac{\tau_{r}-t}{\tau_{r}-\tau_{r-1}}\mu_{\tau_{r-1}}+\dfrac{t-\tau_{r-1}}{\tau_{r}-\tau_{r-1}}\mu_{\tau_r}&\text{if }t\in[\tau_{r-1},\tau_r].
\end{array}
\right.
$$
If $K$ and $K^{(\tau)}$ are the functions defined on $\R_+\times\R$ by $K(t,x)=\displaystyle\int[(y-x)^+-y]\mu_t(dy)$ and $K^{(\tau)}(t,x)=\displaystyle\int[(y-x)^+-y]\mu^{(\tau)}_t(dy)$ respectively, then, for every $x\in\R$
$$
K^{(\tau)}(t,x)=\left\{
\begin{array}{cc}
K(t,x)&\text{if }t\notin[\tau_0,\tau_r]\\
\dfrac{\tau_{1}-t}{\tau_{1}-\tau_0}K(\tau_0,x)+\dfrac{t-\tau_0}{\tau_{1}-\tau_{0}}K(\tau_1,x)&\text{if }t\in[\tau_0,\tau_1]\\
\cdots\cdots\cdots&\cdots\\
\dfrac{\tau_{r}-t}{\tau_{r}-\tau_{r-1}}K(\tau_{r-1},x)+\dfrac{t-\tau_{r-1}}{\tau_{r}-\tau_{r-1}}K(\tau_r,x)&\text{if }t\in[\tau_{r-1},\tau_r].
\end{array}
\right.
$$
Since $(\mu_t,t\geq0)$ is increasing in the WDS order, $K$ is TP$_2$ and, by Proposition \ref{prop:InterPolTP2}, $K^{(\tau)}$ is also TP$_2$.  
\subsubsection{Censoring type transformations }
Let $(\mu_t,t\geq0)$ be a WDS family of integrable probability measures. For every $t\geq0$, we denote the mean of $\mu_t$ by $m_t$.  
Let $a<b$ be fixed real numbers. We consider the family $(\mu^{a,b}_t,t\geq0)$ given by:
\begin{equation}\label{eq:CensorTr}
\mu^{a,b}_t(dy)=(1_{]-\infty,a[}+1_{]b,+\infty[})(y)\mu_t(dy)+\alpha^{a,b}_t\delta_a(dy)+\beta^{a,b}_t\delta_b(dy),
\end{equation}
where $\delta_a$, resp. $\delta_b$ denotes the Dirac measure at point $a$, resp. point $b$, and where
$$
\alpha^{a,b}_t=\frac{1}{b-a}\int_{[a,b]}(b-y)\mu_t(dy)\,\text{ and }\,\beta^{a,b}_t=\frac{1}{b-a}\int_{[a,b]}(y-a)\mu_t(dy).
$$
Let $K$ and $K^{a,b}$ be the functions defined on $\R_+\times\R$ by $K(t,x)=\displaystyle\int[(y-x)^+-y]\mu_t(dy)$ and $K^{a,b}(t,x)=\displaystyle\int[(y-x)^+-y]\mu^{a,b}_t(dy)$ respectively. Then, for every $t\in\R_+$,
$$
K^{a,b}(t,x)=\left\{
\begin{array}{cc}
K(t,x)&\text{if }x\notin[a,b],\\
\dfrac{b-x}{b-a}K(t,a)+\dfrac{x-a}{b-a}K(t,b)&\text{otherwise.}
\end{array}
\right.
$$
We deduce from the WDS ordering property of $(\mu_t,t\geq0)$ and from Proposition \ref{prop:InterPolTP2} that $K^{a,b}$ is TP$_2$. Therefore, according to Theorem \ref{theo:WDSTPTail}, $\left(\mu_t^{a,b},t\geq0\right)$
is non-decreasing in the WDS order.

We turn now to the exposition of some examples of WDS ordered processes.

\section{Some examples of WDS processes}
 We first prove that stochastically non-increasing processes with negative mean are WDS ordered, and then we give some examples of WDS ordered processes which do not decrease stochastically.
\subsection{Stochastically non-increasing processes with negative mean}
The following result states that every stochastically non-increasing process with negative mean increases in the WDS order.  
\begin{theorem}\label{theo:WdsoExa}
Let $(\mu_t,t\geq0)$ be a family of integrable probability measures which have a negative mean and which decreases stochastically. Then $(\mu_t,t\geq0)$ increases in the WDS order.
\end{theorem}
\begin{proof}
Let $(\mu_v,v\geq0)$ be a family of integrable probability measures having  negative means which is stochastically non-increasing. Let $0\leq s\leq t$ and $x\in\R$ be fixed. Since $(\mu_v,v\geq0)$ decreases stochastically, we have $r_{\mu_t}\leq r_{\mu_s}$. If $x\geq r_{\mu_t}$, then $\Psi_{\mu_t}^{wds}(x)=+\infty\geq\Psi_{\mu_s}^{wds}(x)$. Otherwise, we have:
\begin{align*}
\Psi^{wds}_{\mu_t}(x)&=x+\frac{\displaystyle\int(y-x)1_{[x,+\infty[}(y)\mu_t(dy)-\displaystyle\int_{}y\mu_t(dy)}{\displaystyle\int1_{[x,+\infty[}(y)\mu_t(dy)}\\
&=x+\frac{\displaystyle\int(x-y)1_{]-\infty,x]}(y)\mu_t(dy)-x}{\displaystyle\int1_{[x,+\infty[}(y)\mu_t(dy)}.
\end{align*}
Since $(\mu_v,v\geq0)$ is stochastically non-increasing, 
$$
\int(x-y)1_{]-\infty,x]}(y)\mu_t(dy)-x\geq\int(x-y)1_{]-\infty,x]}(y)\mu_s(dy)-x>0 
$$
and
$$
0< \int1_{[x,+\infty[}(y)\mu_t(dy)\leq\int1_{[x,+\infty[}(y)\mu_s(dy).
$$
We then deduce that $\Psi^{wds}_{\mu_t}(x)\geq \Psi^{wds}_{\mu_s}(x)$.
\end{proof}  
\begin{remark}
We deduce from the preceding result that the Cox-Hobson algorithm provides a way to associate a Markovian process to a given stochastically non-decreasing family of distributions with densities. Indeed, let $(\mu_t,t\geq0)$ be a stochastically non-decreasing family of distributions with densities; if $m_0$ is a real number such that $m_0<\int y\mu_0(dy)$ and if $\mu_t^{g}$ denotes the image of $\mu_t$ under $g:y\longmapsto m_0-y$, then $(\mu^{g}_t,t\geq0)$ is a stochastically non-increasing family of distributions with negative means; let $(B_t,t\geq0)$ be a standard Brownian motion started at $0$, and let $T^{g}_t$ be the Cox-Hobson stopping time attached to $\mu^{g}_t$; then $\left(B_{T^{g}_t},t\geq0\right)$ is a Markovian supermartingale associated to $\left(\mu^{g}_t,t\geq0\right)$; we deduce that $\left(m_0-B_{T^{g}_t},t\geq0\right)$ is a Markovian submartingale associated to $(\mu_t,t\geq0)$. Note that the quantile process associated to a stochastically non-decreasing process is not necessarily Markovian. Recently, Juillet \cite[Proposition 4.4]{Ju} characterized Markovian quantile processes. 
\end{remark} 

\subsection{Some examples of WDS processes which do not decrease stochastically}
Since Theorem \ref{theo:WdsoExa} states that, for processes with negative mean, the decreasing stochastic order implies the WDS order, one may ask whether the reverse implication holds. We give a negative answer to this statement by providing two families of WDS processes  which do not decrease stochastically. We start by the following discrete example. 
\begin{prop}
Let $k\geq1$ and let $(\mu_t,t\in]0,1[)$ be the family of probability measures given by:
$$
\mu_t(dy)=(1-t^{k}+t^{k+1})\delta_{-t^k}(dy)+(t^k-t^{k+1})\delta_{1-t^k}(dy),
$$
where, for every $a\in\R$, $\delta_a$ denotes the Dirac measure at point $a$. Then $(\mu_t,t\in]0,1[)$ is a WDS ordered process which does not decrease stochastically.
\end{prop}
\begin{proof}
Let $k\geq1$ and $t\in]0,1[$. Then, for every $x\in\R$,
$$
\mu_t([x,+\infty[)=\left\{
\begin{array}{cl}
1&\text{ if }x\leq-t^k,\\
t^k-t^{k+1}&\text{ if }-t^k<x\leq1-t^k,\\
0&\text{ if }x>1-t^k.
\end{array}
\right.
$$
In particular, $\mu_t([0,+\infty[)=t^k-t^{k+1}$ which proves that $(\mu_t,t\in]0,1[)$ does not decrease stochastically. Moreover,
$$
\int_{[x,+\infty[}y\mu_t(dy)=\left\{
\begin{array}{cl}
-t^{k+1}&\text{ if }x\leq-t^k,\\
(t^k-t^{k+1})(1-t^k)&\text{ if }-t^k<x\leq1-t^k,\\
0&\text{ if }x>1-t^k.
\end{array}
\right.
$$
Hence, for every $x\in\R$,
$$
\Psi^{wds}_{\mu_t}(x)=\left\{
\begin{array}{cl}
0&\text{ if }x\leq-t^k,\\
1-t^k+\dfrac{t}{1-t}&\text{ if }-t^k<x<1-t^k,\\
+\infty&\text{ if }x\geq1-t^k
\end{array}
\right.
$$
which shows that the family $\left(\Psi^{wds}_{\mu_t},t\in]0,1[\right)$ is pointwise non-decreasing. 
\end{proof}
The next result gives another interesting family of non stochastically decreasing WDS ordered processes
\begin{prop} 
For every $k>-1$, the family $\left(\mu_t,t\in\left]1/2,2\right[\right)$ of probability measures given by
\[
\mu_t(dy)=\frac{2-t}{2\alpha(t)}1_{[-\alpha(t),0[}(y)dy+\frac{(k+1)t^{k+2}}{2}y^k1_{[0,1/t[}(y)dy,
\]
where
\[
\forall\,t\in\left]1/2,2\right[,\text{ }
\alpha(t)=\frac{k+1}{k+2}\frac{(2t)^{k+2}(2t+1)}{2-t}
\]
is a WDS ordered family which does not decrease  stochastically.
\end{prop}
\begin{proof}
Observe that $\alpha$ is positive and non-decreasing on $]1/2,2[$ and that 
\[
\forall\,t\in\left]1/2,2\right[,\text{ }m_t:=\int_{\R}y\mu_t(dy)=\frac{k+1}{2(k+2)}\left(1-\frac{1}{2}(2t)^{k+2}(2t+1)\right)<0.
\]
Fix $1/2<s\leq t<2$ and $x\in\R$. If $x\geq1/t$, then $\Psi^{wds}_t(x)=+\infty\geq\Psi^{wds}_s(x)$, and if $x\leq-\alpha(s)$, then $\Psi^{wds}_s(x)=0\leq\Psi^{wds}_t(x)$. Suppose that $x\in\left[0,1/t\right[$. Then for every $v\in[s,t]$,
\begin{equation}\label{eq:exastarMRLSI}
\int_x^{1/v}\mu_v(dy)=\frac{(k+1)v^{k+2}}{2}\int_x^{1/v}y^kdy=\frac{v}{2}\left(1-(vx)^{k+1}\right)
\end{equation}
and
\begin{equation}\label{eq:exastarMRLMean}
\int_x^{1/v}y\mu_v(dy)=\frac{(k+1)v^{k+2}}{2}\int_x^{1/v}y^{k+1}dy=\frac{k+1}{2(k+2)}\left(1-(vx)^{k+2}\right).
\end{equation}
Note that the equality (\ref{eq:exastarMRLSI}) shows that  $\left(\mu_v,v\in\left]1/2,2\right[\right)$ does not decrease in the stochastic order. We deduce from (\ref{eq:exastarMRLSI}) and (\ref{eq:exastarMRLMean}) that
\begin{align*}
\Psi^{wds}_v(x)&=\frac{k+1}{k+2}\times\frac{1-(vx)^{k+2}-1+(2v)^{k+1}v(2v+1)}{v(1-(vx)^{k+1})}\\
&=\frac{k+1}{k+2}\times\frac{v^{k+1}}{1-(vx)^{k+1}}(2^{k+1}(2v+1)-x^{k+2}).
\end{align*}
On the other hand, 
$$
v\longmapsto \frac{v^{k+1}}{1-(vx)^{k+1}}
$$
is a nonnegative non-decreasing function on $\left]1/2,2\right[$ and, since $x\leq1/v$ and $v\geq1/2$, 
$$
 v\longmapsto 2^{k+1}(2v+1)-x^{k+2} 
$$ 
is also a nonnegative non-decreasing function on $\left]1/2,2\right[$. Therefore, $v\longmapsto\Psi^{wds}_v(x)$ is non-decreasing on $\left]1/2,2\right[$ and, as a consequence, $\Psi_s^{wds}(x)\leq\Psi_t^{wds}(x)$.\\
Suppose now that $x\in]-\alpha(s),0[$. For every $v\in[s,t]$, $-\alpha(v)\leq-\alpha(s)$ since $\alpha$ is non-decreasing, and we have:
\begin{equation*}
\int_x^{1/v}\mu_v(dy)=\int_x^0\mu_v(dy)+\int_0^{1/v}\mu_v(dy)=-\frac{(k+2)(2-v)^2}{2(k+1)(2v)^{k+2}(2v+1)}x+\frac{v}{2}.
\end{equation*}
and 
\begin{equation*}
\int_x^{1/v}y\mu_v(dy)=\int_x^{0}y\mu_v(dy)+\int_0^{1/v}y\mu_v(dy)=-\frac{(k+2)(2-v)^2}{4(k+1)(2v)^{k+2}(2v+1)}x^2+\frac{k+1}{2(k+2)}.
\end{equation*}
We then obtain
\begin{equation*}
\Psi_v^{wds}(x)=\frac{-\dfrac{(2-v)^2}{(2v)^{k+3}(2v+1)}x^2+ \dfrac{(k+1)^2}{(k+2)^2}(2v)^{k+1}(2v+1)}{-\dfrac{2(2-v)^2}{(2v)^{k+3}(2v+1)}x+1 }.
\end{equation*}
Remark that, for every $x\in]-\alpha(s),0[$,
\begin{equation*}
v\longmapsto -\dfrac{2(2-v)^2}{(2v)^{k+3}(2v+1)}x+1  
\end{equation*}
is non-increasing and nonnegative on $]1/2,2[$ and that, since $-x^2>-\alpha^2(s)\geq-\alpha^2(v)$, 
\begin{equation*}
v\longmapsto -\dfrac{(2-v)^2}{(2v)^{k+3}(2v+1)}x^2+ \dfrac{(k+1)^2}{(k+2)^2}(2v)^{k+1}(2v+1) 
\end{equation*}
is non-decreasing and nonnegative on $]1/2,2[$. Consequently, $v\longmapsto\Psi_v^{wds}(x)$ is non-decreasing on $]1/2,2[$ and, hence, $\Psi_s^{wds}(x)\leq\Psi_t^{wds}(x)$. 
\end{proof}

\section*{Acknowledgement}
We thank the anonymous referees for valuable suggestions and comments which have improved the manuscript. We also thank the African Center of Excellence in Technologies, Information and Communication (CETIC) to placed in our disposal its infrastructures. This help use to improve conditions of this work.

\end{document}